\documentclass[11pt,oneside,reqno,a4paper]{amsart}
\usepackage{amsmath}
\usepackage{verbatim}
\usepackage[mathcal]{eucal}

\topmargin=5pt\usepackage{hyperref}
 \makeatletter 
\newtheorem{definition}{Definition}
\newcommand{\brdef}{\begin{definition}}
\newcommand{\erdef}{\end{definition}}
\newtheorem{corollary}{Corollary}
\newcommand{\bcorollary}{\begin{corollary}}
\newcommand{\ecorollary}{\end{corollary}}

\newtheorem{theorem}{Theorem}
\newcommand{\bth}{\begin{theorem}}
\newcommand{\eth}{\end{theorem}}
\newtheorem{lemma}{Lemma}
\newcommand{\ble}{\begin{lemma}}
\newcommand{\ele}{\end{lemma}}

\newlength{\defbaselineskip}
\newcommand{\ds}{\displaystyle}
\newcommand{\mC}{{\mathbb C}}

\newcommand{\mU}{{\mathbb U}}
\newcommand{\cA}{{\mathcal A}}

\newcommand{\cS}{{\mathcal S}}

\begin{document}
\title[Univalent functions with positive
coefficients...]
{Univalent functions with positive
coefficients \\involving Touchard polynomials}
 \author{G.MURUGUSUNDARAMOORTHY$^{1}$ and Saurabh Porwal$^{2,\; \ast}$ }
 \maketitle
\begin{center}

$^{1}$ School of  Advanced Sciences,
\\VIT University,\\ Vellore - 632014, India. \\
{ email:gmsmoorthy@yahoo.com}
\end{center}
\begin{center}
${}^{2,\; \ast}$Department of Mathematics \\
Ram Sahai Government Degree College,\\
 Bairi-Shivrajpur-Kanpur-209205, (U.P.), India\end{center}
\begin{abstract}The purpose of the present paper is to establish connections between various subclasses of analytic
univalent functions by applying certain convolution operator involving Touchard polynomials. To be more precise, we investigate
such connections with the classes of analytic univalent functions
with positive coefficients in the open
unit disk $\mathbb{U}.$
 \
\\
\
\\
2000 Mathematics Subject Classification: 30C45. \
\\
\
\\
{\it Keywords and Phrases}: Univalent, Starlike functions, Convex functions,
Hadamard product, Touchard polynomials.
series.
\end{abstract}
\maketitle
\section{INTRODUCTION}

\par The application of special function on Geometric Function Theory is a current and interesting topic of research. It is frequently applied in various branches Mathematics, Physics, Sciences, Engineering and Technology. The surprising use of generalized hypergeometric function by L. de Branges \cite{deBranges} in the solution  of the famous Bieberbach
conjecture.  There is an extensive
literature dealing with analytical and geometric properties of various
types of special functions, especially for the generalized,
Gaussian hypergeometric functions \cite{cho,merkes,mostafa,silver,HMSSS}.
\par The Touchard polynomials, studied by Jacques Touchard \cite{05}, also called the exponential generating polynomials (see \cite{02}, \cite{03}, \cite{04}) or Bell polynomials (see \cite{00}) comprise a polynomial sequence of binomial type that for $ X $ is a random variable with a Poisson distribution with expected value $\ell$, then its nth moment is $E(X_{\kappa}) = \mathcal{T}(\kappa,\ell)$, leading to the form:
\begin{equation}\label{GMST1}
  \mathcal{T}(\kappa,\ell)= e^{\kappa}\sum\limits_{n=0}^{\infty} \frac{\kappa^{n}n^{\ell}}{n!}z^n
\end{equation}
Touchard polynomials coefficients after the second force as following
\par Lately,  introduce Touchard polynomials coefficients after the second force as following
\begin{equation}\label{PHI}
\Phi_{\kappa}^{\ell}(z)=z+\sum\limits_{n=2}^{\infty} \frac{(n-1)^{\ell}\kappa^{n-1}}{(n-1)!}e^{-\kappa}z^n,\qquad z\in \mathbb{D},\end{equation} where $\ell\geq 0;\kappa>0 $ and we note that, by ratio test the radius of convergence of above series is infinity.
Let $\mathcal H$ be the class of functions analytic in the unit disk
$\mathbb U=\{z\in \mathbb C: |z|<1\}$. Let $\cA$ be the class of
functions $f\in\mathcal H$ of the form
\begin{equation}\label{eq1.1}
   f(z) = z + \sum_{n=2}^\infty a_n z^n,\ \ \ z\in\mathbb U.
\end{equation}
 We also let $\cS$  be the subclass of $\mathcal{A}$ consisting of
functions which are  normalized by $f(0)=0=f'(0)-1$ and also
univalent in $\mathbb{U}.$
\par Denote by  $\mathcal{V}$ the subclass of $\mathcal{A}$ consisting of functions of the form\begin{equation}\label{eq1.5}
 f(z) = z + \sum_{n=2}^\infty a_n z^n , a_n\geq 0.\end{equation}
For functions $f\in \mathcal{A}$ given by \eqref{eq1.1} and $g\in
\mathcal{A}$ given by $g(z)=z + \sum_{n=2}^{\infty}b_n z^n, $ we
define the Hadamard product (or convolution) of $f$ and $g$ by
\begin{equation}\label{e1.2}
  (f*g)(z) = z + \sum\limits_{n=2}^{\infty} a_n b_n z^n, \,\,\,\, z
  \in \mU.
\end{equation}

Now, we define the linear operator
\begin{equation*}
    \mathcal{I}(l,m,z):\mathcal{A}\rightarrow\mathcal{A}
\end{equation*}
defined by the convolution or hadamard product
\begin{equation}\label{I}
  \mathcal{ I}(l,m,z)f=\Phi_{m}^{\ell}(z)\ast f(z)
=z+\sum_{n=2}^{\infty}\frac{(n-1)^{l}m^{n-1}}{(n-1)!}e^{-m}a_nz^n,
\end{equation}
where $\Phi_{m}^{\ell}(z)$ is the  series given by (\ref{PHI}).
\par The class $\mathcal M(\alpha)$ of starlike functions of order
$1< \alpha \leq\frac{4}{3}$
\begin{equation*}
    {\mathcal M}(\alpha):=\left\{ f\in {\mathcal A}:\ {\rm \Re}\frac{zf'(z)}{f(z)}
    <\alpha ,\ z\in\mathbb U\right \}
\end{equation*}and
the class ${\mathcal N}(\alpha)$ of convex functions of order $1< \alpha \leq\frac{4}{3}$
\begin{align*}
   {\mathcal N}(\alpha):&
   =\left\{ f\in {\mathcal A}:\ {\rm \Re
   }\left(1+\frac{zf''(z) }{f'(z)}\right)
   <\alpha ,\ z\in\mathbb U\right \}= \left\{f\in \mathcal{A}:\ zf'\in
   {\mathcal M}{(\alpha) }\right\}
\end{align*}
were introduced by Uralegaddi et al.\cite{ural1} (see\cite{DC,Dp}).Also let ${\mathcal M}^*(\alpha)\equiv{\mathcal M}(\alpha)\cap \mathcal{V}$
and ${\mathcal N}^*(\alpha)\equiv{\mathcal N}(\alpha)\cap \mathcal{V}.$

\par In this paper we introduce two new subclasses of $\mathcal{S}$ namely
 $\mathcal{M}(\lambda,\alpha)$ and $\mathcal{N}(\lambda,\alpha)$  to discuss some inclusion properties.

\par For some $\alpha$~$(1< \alpha \leq\frac{4}{3})$ and $ \lambda$($0 \leq \lambda<1)$,
we let    $\mathcal{M}(\lambda,\alpha)$ and
$\mathcal{N}(\lambda,\alpha)$  be two new subclass of $\mathcal{S}$
consisting of functions of the form \eqref{eq1.1} satisfying the
analytic criteria
\begin{equation}\label{e0.1}
   \mathcal{M}(\lambda,\alpha) := \left\{ f \in \cS :\Re\left(\frac{zf'(z)}{(1-\lambda)f(z)+\lambda zf'(z)} \right)
  < \alpha,\, \, z \in \mathbb{U} \right\}.
\end{equation}
\begin{equation}\label{e0.0}
    \mathcal{N}(\lambda,\alpha) := \left\{ f \in \cS :\Re\left(\frac{f'(z)+ zf''(z) }{f'(z)
    + \lambda zf''(z)} \right) < \alpha,\, \, z \in
    \mathbb{U}\right\}.
\end{equation}
We also  let $ \mathcal{M}^*(\lambda,\alpha)\equiv \mathcal{M}(\lambda,\alpha)\cap \mathcal{V}$ and $\mathcal{N}^*(\lambda,\alpha)\equiv \mathcal{N}(\lambda,\alpha)\cap \mathcal{V}$.
\par Note that   $\mathcal{M}(0,\alpha)= \mathcal{M}(\alpha),$~~$\mathcal{N}(0,\alpha)=\mathcal{N}(\alpha); $ $\mathcal{M}^*(\alpha)$~~
and $\mathcal{N}^*(\alpha)$~~ the subclasses of  studied by Uralegaddi et al.\cite{ural1}.


\par  Motivated by
results on connections between various subclasses of analytic
univalent functions by using hypergeometric functions (see
\cite{cho,merkes,mostafa,silver,HMSSS}), we obtain necessary and sufficient condition
for function $\Phi(l,m,z)$ to be in the classes
$\mathcal{M}(\lambda,\alpha)$ , $\mathcal{N}(\lambda,\alpha) $ and connections between
$\mathcal{R}^\tau (A,B)$ by applying convolution operator.
\section{Preliminary Results}
To prove our main results we shall require the following definitions and lemmas.
\begin{definition}
The $l^{th}$ moment of the Poisson distribution is defined as
\[
\mu_l^{'}=\sum_{n=0}^{\infty}\frac{n^lm^n}{n!}e^{-m}.\]
\end{definition}


\begin{lemma}\label{gms1}
For some $\alpha$~$(1< \alpha \leq\frac{4}{3})$ and $ \lambda(0 \leq \lambda<1)$, and if $f\in\mathcal{V}$ then $f\in\mathcal{M}^*(\lambda,\alpha)$ if and only if
\begin{equation}\label{2t1}
   \sum_{n=2}^\infty[n-(1+n\lambda-\lambda)\alpha]|a_n|\leq \alpha-1.
\end{equation}
\end{lemma}
\begin{lemma}\label{gms1b}
For some $\alpha$~$(1< \alpha \leq\frac{4}{3})$ and $ \lambda(0 \leq \lambda<1)$, and if $f\in\mathcal{V}$ then $f\in\mathcal{N}^*(\lambda,\alpha)$ if and only if
\begin{equation}\label{eqbreaz}
   \sum_{n=2}^{\infty}n[n-(1+n\lambda-\lambda)\alpha]a_n\leq \alpha-1 .
\end{equation}
\end{lemma}
\section{Main Results}
\par For convenience throughout in the sequel, we use
the following notations:
\begin{eqnarray}
       \sum_{n=2}^{\infty} \frac{m^{n-1}}{(n-1)!} &=& e^m-1 \\
        \sum_{n=2}^{\infty} \frac{m^{n-1}}{(n-2)!} &=& me^m \\
       \sum_{n=2}^{\infty} \frac{m^{n-1}}{(n-3)!} &=& m^2e^m \\
       \end{eqnarray}


\begin{theorem}\label{gms1a}
If  $ m>0(m\neq0,-1,-2,\dots)$, $l\in N_{0}$ then $\Phi(l,m,z)\in\mathcal{M}^*(\lambda,\alpha)$ if and only if
\begin{equation}\label{c1}
{e^{- m}}\left\{ {\begin{array}{*{20}{c}}
   {\left( 1 - \alpha \lambda \right)\mu _{l + 1}^{'} + \left( 1 - \alpha  \right)\mu _{l}^{'},~~~ l \ge 1}  \\
   {\left( {1 - \alpha \lambda } \right)m{e^m} + \left( {1 - \alpha } \right)\left( {{e^{m}} - 1} \right), ~~~ l = 0}  \\
 \end{array} } \right. \leq \alpha-1.
\end{equation}
\end{theorem}
\begin{proof}
To prove that $\Phi(l,m,z)\in\mathcal{M}^*(\lambda,\alpha)$, then by virtue of Lemma \ref{gms1}, it suffices to show that

\begin{equation}\label{thm4e11}
    \sum_{n=2}^{\infty} [n-(1+n\lambda-\lambda)\alpha]\frac{(n-1)^lm^{n-1}}{(n-1)!}e^{-m}\leq \alpha-1.
\end{equation}
Now
\begin{eqnarray*}
&~& e^{-m}\sum_{n=2}^{\infty} [n(1-\lambda\alpha)-\alpha(1-\lambda)]\frac{(n-1)^lm^{n-1}}{(n-1)!}\\
      &=&e^{-m}\sum_{n=2}^{\infty}[(n-1)(1-\lambda\alpha)+1-\alpha]\frac{(n-1)^lm^{n-1}}{(n-1)!}\\
  &=&e^{-m}\sum_{n=1}^{\infty}[n(1-\lambda\alpha)+1-\alpha]\frac{n^lm^{n}}{n!}\\
&=&e^{-m}\sum_{n=1}^{\infty}[(1-\lambda\alpha)\frac{n^{l+1}m^{n}}{n!}+(1-\alpha)\frac{n^lm^{n}}{n!}]\\
&=&{e^{- m}}\left\{ {\begin{array}{*{20}{c}}
   {\left( 1 - \alpha \lambda \right)\mu _{l + 1}^{'} + \left( 1 - \alpha  \right)\mu _{l}^{'},~~~ l \ge 1}  \\
   {\left( {1 - \alpha \lambda } \right)m{e^m} + \left( {1 - \alpha } \right)\left( {{e^{m}} - 1} \right), ~~~ l = 0}  \\
 \end{array} } \right.\end{eqnarray*}

But this expression is bounded above by $\alpha-1$ if and only if
\eqref{c1} holds. Thus the proof is complete.\end{proof}


\begin{theorem}\label{gms2a}
If  $ m>0(m\neq0,-1,-2,\dots)$, $l\in N_{0}$ then $\Phi(l,m,z)\in\mathcal{N}^*(\lambda,\alpha)$ if and only if
\begin{equation}\label{c2}
{e^{- m}}\left\{ {\begin{array}{*{20}{c}}
   {\left( 1 - \alpha \lambda \right)\mu _{l + 2}^{'}+\left( 2- \alpha \lambda-\alpha \right)\mu _{l + 1}^{'} + \left( 1 - \alpha  \right)\mu _{l}^{'},~~~ l \ge 1}  \\
   {\left( {1 - \alpha \lambda } \right)(m^2+m){e^m} +\left( 2- \alpha \lambda-\alpha \right)me^m+ \left( {1 - \alpha } \right)\left( {{e^{m}} - 1} \right), ~~~ l = 0}  \\
 \end{array} } \right. \leq \alpha-1.
\end{equation}
\end{theorem}
\begin{proof}
To prove that $\Phi(l,m,z)\in\mathcal{N}^*(\lambda,\alpha)$, then by virtue of Lemma \ref{gms1b}, it suffices to show that

\begin{equation}\label{thm4e11}
    \sum_{n=2}^{\infty}n[n-(1+n\lambda-\lambda)\alpha]\frac{(n-1)^lm^{n-1}}{(n-1)!}e^{-m}\leq \alpha-1.
\end{equation}
Now
\begin{eqnarray*}
&~& e^{-m}\sum_{n=2}^{\infty} n[n(1-\lambda\alpha)-\alpha(1-\lambda)]\frac{(n-1)^lm^{n-1}}{(n-1)!}\\
      &=&e^{-m}\sum_{n=2}^{\infty}[(1-\lambda\alpha)(n-1)^2+(2-\alpha\lambda-\alpha)(n-1)+1-\alpha]\frac{(n-1)^lm^{n-1}}{(n-1)!}\\
  &=&e^{-m}\sum_{n=2}^{\infty}[(1-\lambda\alpha)\frac{(n-1)^{l+2}m^{n-1}}{(n-1)!}+(2-\alpha\lambda-\alpha)\frac{(n-1)^{l+1}m^{n-1}}{(n-1)!}+(1-\alpha)\frac{(n-1)^lm^{n-1}}{(n-1)!}]\\
&=&{e^{- m}}\left\{ {\begin{array}{*{20}{c}}
   {\left( 1 - \alpha \lambda \right)\mu _{l + 2}^{'}+\left( 2- \alpha \lambda-\alpha \right)\mu _{l + 1}^{'} + \left( 1 - \alpha  \right)\mu _{l}^{'},~~~ l \ge 1}  \\
   {\left( {1 - \alpha \lambda } \right)(m^2+m){e^m} +\left( 2- \alpha \lambda-\alpha \right)me^m+ \left( {1 - \alpha } \right)\left( {{e^{m}} - 1} \right), ~~~ l = 0}  \\
 \end{array} } \right.\end{eqnarray*}

But this expression is bounded above by $\alpha-1$ if and only if
\eqref{c2} holds. Thus the proof is complete.\end{proof}


\setcounter{equation}{0}
\section{Inclusion Properties}
\par A function $f \in \cA$ is said to be in the class $\mathcal{R}
^{\tau } (A, B)$, $(\tau \in \mC \backslash \{0\},\,\, -1 \leq B <
A \leq 1)$,  if it satisfies the inequality
\begin{equation*}
  \left|\ds  \frac{f^\prime(z)  - 1} {\ds (A-B)\tau - B[f^\prime(z) -
  1]} \right| < 1\ \ \ (z \in \mU).
\end{equation*}
The class $\mathcal{R} ^{\tau } (A, B)$ was introduced earlier by
Dixit and Pal \cite{dix}.\\ It is of interest to note that if
\begin{equation*}
   \tau = 1,\,\, A = \beta \mbox{ and }\,\, B = -\beta \,\,(0 <
   \beta \leq 1),
\end{equation*} we obtain the class of functions $f
\in \cA$ satisfying the inequality
\begin{equation*}
   \left| \frac{f^\prime(z)-1} {f^\prime(z)+1} \right| < \beta \quad (z
   \in \mU)
\end{equation*}
which was studied by
(among others) Padmanabhan \cite{padma} and  Caplinger and Causey
\cite{cap}.
\begin{lemma}\cite{dix}\label{s2th4} If
$f\in \mathcal{R} ^{\tau } (A,B)$ is of form \eqref{eq1.1}, then
\begin{equation}\label{dixcondi}
   \left| a_{n} \right| \leq (A-B)\frac{\left| \tau \right| }{n}, \quad
   n \in \mathbb{N} \setminus\{1\}.
\end{equation}
The result is sharp.
\end{lemma}

Making use of the  Lemma \ref{s2th4} we will study the action of the Poissons distribution series
 on the class $\mathcal{M}(\lambda,\alpha).$
\begin{theorem}\label{thm2}
Let  $  m>0(m\neq0,-1,-2,\dots),$ $l\in N_{0}$.  If $ f \in
\mathcal{R}^{\tau}(A,B)$,then  $ \mathcal{I}(l,m,z)f \in \mathcal{N}^*(\lambda,\alpha)$ if and only if
\begin{equation}\label{e0.5}
  {(A-B)|\tau|e^{- m}}\left\{ {\begin{array}{*{20}{c}}
   {\left( 1 - \alpha \lambda \right)\mu _{l + 1}^{'} + \left( 1 - \alpha  \right)\mu _{l}^{'},~~~ l \ge 1}  \\
   {\left( {1 - \alpha \lambda } \right)m{e^m} + \left( {1 - \alpha } \right)\left( {{e^{m}} - 1} \right), ~~~ l = 0}  \\
 \end{array} } \right.\\
\le \alpha-1.
\end{equation}

\end{theorem}
\begin{proof}
Let $f$ be of the form \eqref{eq1.1} belong to the class
$\mathcal{R}^{\tau}(A,B).$ By virtue of Lemma \ref{gms1b}, it
suffices to show that
\begin{equation*}
   \sum_{n=2}^{\infty} n [n-(1+n\lambda-\lambda)\alpha]\frac{(n-1)^lm^{n-1}}{(n-1)!}|a_n|\leq \alpha-1.
\end{equation*} Let
Now
\begin{eqnarray*}
&~& e^{-m}\sum_{n=2}^{\infty} [n(1-\lambda\alpha)-\alpha(1-\lambda)]\frac{(n-1)^lm^{n-1}}{(n-1)!}|a_n|\\
      &\le(A-B)|\tau|&e^{-m}\sum_{n=2}^{\infty}[(n-1)(1-\lambda\alpha)+1-\alpha]\frac{(n-1)^lm^{n-1}}{(n-1)!}\\
  &=&(A-B)|\tau|e^{-m}\sum_{n=1}^{\infty}[n(1-\lambda\alpha)+1-\alpha]\frac{n^lm^{n}}{n!}\\
&=&(A-B)|\tau|e^{-m}\sum_{n=1}^{\infty}[(1-\lambda\alpha)\frac{n^{l+1}m^{n}}{n!}+(1-\alpha)\frac{n^lm^{n}}{n!}]\\
&=&{(A-B)|\tau|e^{- m}}\left\{ {\begin{array}{*{20}{c}}
   {\left( 1 - \alpha \lambda \right)\mu _{l + 1}^{'} + \left( 1 - \alpha  \right)\mu _{l}^{'},~~~ l \ge 1}  \\
   {\left( {1 - \alpha \lambda } \right)m{e^m} + \left( {1 - \alpha } \right)\left( {{e^{m}} - 1} \right), ~~~ l = 0}  \\
 \end{array} } \right.\\
&\le& \alpha-1.\end{eqnarray*}

 \end{proof}
\begin{theorem}\label{SMJ1}
Let  $  m>0(m\neq0,-1,-2,\dots)$, $l\in N_{0}$ then
$\mathcal{L}(l,m,z)=\int_{0}^{z}\frac{\mathcal{I}(l,m,t)}{t}dt$ is in
$\mathcal{N}^*(\lambda,\alpha)$ if and only if inequality \eqref{c1} is satisfied.
\end{theorem}
\begin{proof}
Since
\begin{equation*}
   \mathcal{L}(l,m,z)=
   z+\sum_{n=2}^{\infty}\frac{(n-1)^lm^{n-1}}{(n-1)!}e^{-m}~\frac{z^n}{n}.
\end{equation*}
By virtue of Lemma \ref{gms1a}, it
suffices to show that
\begin{equation*}
   \sum_{n=2}^{\infty}  n[n-(1+n\lambda-\lambda)\alpha]
   \frac{(n-1)^lm^{n-1}}{n(n-1)!}e^{-m}\leq\alpha-1.
\end{equation*}
Now,
\begin{eqnarray*}
\sum_{n=2}^{\infty}   n[n-(1+n\lambda-\lambda)\alpha]
    \frac{(n-1)^lm^{n-1}}{n(n-1)!}e^{-m}=\sum_{n=2}^{\infty} [n-(1+n\lambda-\lambda)\alpha] \frac{(n-1)^lm^{n-1}}{(n-1)!}e^{-m}.\\
\end{eqnarray*}
Proceeding as in Theorem \ref{gms1a} we obtain the required result.
 \end{proof}
\section{Open Problems}
It is interesting to find the result of Theorems \ref{gms1a}-\ref{SMJ1}, when $l$ is a real number.



\end{document}